\documentclass[a4paper]{amsart}
\usepackage{leqno,latexsym}
\usepackage[german,french,english]{babel}
\usepackage{pgfplots}
\usepackage[colorlinks,pdftex]{hyperref}
\hypersetup{
    colorlinks,%
    citecolor=blue,%
    filecolor=blue,%
    linkcolor=blue,%
    urlcolor=blue
}
\usepackage{doi}

\newtheorem{thm}{Theorem}
\newtheorem{lem}[thm]{Lemma}
\newtheorem{prop}[thm]{Proposition}
\theoremstyle{definition}

\newtheorem{rem}[thm]{Remark}
\newtheorem{ex}[thm]{Example}

\long\def\comment#1\endcomment{}

\begin{document}


\title{On unsolvable equations of prime degree}

\author{Juliusz Brzezi\'nski}
\address{Department of Mathematical Sciences, Chalmers University of
Technology and University of Gothenburg.
SE 412 96 Gothenburg, Sweden} 
\email{jub@chalmers.se}

\author{Jan Stevens}
\address{Department of Mathematical Sciences, Chalmers University of
Technology and University of Gothenburg.
SE 412 96 Gothenburg, Sweden} 
 \email{stevens@chalmers.se}

\begin{abstract}Leopold Kronecker observed that either all roots or only one root
of a solvable irreducible equation of odd prime degree with integer coefficients
are real. This gives a possibility to construct specific examples of equations
not solvable by radicals.
A relatively elementary proof without using the full power of Galois theory is due to
Heinrich Weber.  We give a rather short proof of Kronecker's Theorem with an argument that is slightly different from Weber's.
Several modern presentations of Weber's proof
contain inaccuracies, which can be traced back to an error in the original proof.
We discuss this error and how it can be corrected.
\end{abstract}



\maketitle

\section*{Introduction}

One of the main objectives of any course in Galois theory is a proof of the existence of polynomials that can not be solved by radicals, showing that
 similar formulas as for quadratic, cubic, or quartic equations cannot exist for  higher degree equations.
Today   this kind of  problem mainly has  a pedagogical and historical value. Galois theory is an extremely important part of mathematics and therefore one of the fundamental parts of mathematical education. Solvability of  polynomial equations is used 
as an illustration of the effectivity of abstract algebraic methods in solving concrete, important, and interesting mathematical problems.  The fact that
most equations cannot be solved  by radicals
is usually presented using the full power of Galois theory, that is, the Galois correspondence and the relations between  solvable equations and solvable groups.
A hint that a more elementary approach is possible, is given  by the circumstance that the
original impossibility  proof of  Niels-Henrik \href{https://zbmath.org/authors/abel.niels-henrik}{Abel}
predates the work of \'Evariste \href{https://zbmath.org/authors/galois.evariste}{Galois}.

Specific examples of unsolvable equations are easily obtained from the fact that
for an irreducible equation of odd prime degree with integer coefficients
all of its roots or only one are real, an observation made by
Leopold \href{https://zbmath.org/authors/kronecker.leopold}{Kronecker} \cite{K2}.
A  relatively elementary proof, using rather limited knowledge related to the field extensions,
was
given by Heinrich \href{https://zbmath.org/authors/weber.heinrich-martin}{Weber} \cite{WW1}.

Weber's proof occurs in the algebra textbook \cite{N} by
Trygve \href{https://zbmath.org/authors/nagell.trygve}{Nagell}.
Thanks to this book
in Swedish, the first Author  became aware of the
possibility to prove the unsolvability of the quintic  in an ``elementary'' way and
transformed it into  Exercise 13.6 in \cite{B}.  Unfortunately, the included solution
is incorrect.  We were led to a closer study of Nagell's proof and to a search
for its origin. We found several recent presentations of Weber's proof in textbooks
\cite{Pr, MJP}, which are based on the English translation \cite{ DE}
of the  popular account
by Heinrich \href{https://zbmath.org/authors/dorrie.heinrich}{Dörrie} \cite{D}.
Several objections to \cite{DE} and expositions following it have been raised,
see \cite{PC,Sk}.
The ultimate source for the  inaccuracies observed is an error in Weber's proof.
The problem lies in reducible radicals (radicals of the form $\alpha=\sqrt[q]a$ with
$X^q-a$ reducible). The need to deal with them also
complicates the variant of the proof we developed.

In Section 1, we place Kronecker's Theorem in a historical context. We discuss Kronecker's arguments for the theorem and trace the history
of ``elementary'' proofs. In Section 2, we discuss the criteria that make the proofs of Kronecker's Theorem elementary. We also recall a number of algebraic results that, according to these criteria, are needed in its elementary proofs.  Section 3 contains our proof of Kronecker's  Theorem. Finally, in Section 4, we discuss the
mistake in Weber's proof and in proofs inspired by his idea, as well as ways in which
the proofs can be corrected.

\section{History}
After Gerolamo \href{https://zbmath.org/authors/cardano.gerolamo}{Cardano}
published the solution of cubic and quartic equations in 1545,
attempts were made to find solutions of quintic equations, expressing their roots
in a similar way as function of the coefficients, by formulas involving only  arithmetical operations and radicals.
The work of E.~\href{https://zbmath.org/authors/waring.edward}{Waring},
A.-T.~\href{https://zbmath.org/authors/vandermonde.a-t}{Vandermonde} and especially
J.-L.~\href{https://zbmath.org/authors/lagrange.joseph-louis}{Lagrange} around 1770
suggested  that such formulas for equations of degree 5 probably do not exist.
This was proven to be so in the work of
 Paulo 
\href{https://zbmath.org/authors/ruffini.paolo}{Ruffini}
and Niels Henrik Abel between 1799 and 1826, see the expository paper by
Michael \href{https://zbmath.org/authors/rosen.michael-i}{Rosen}
\cite{R}.

Much of the effort of  Abel, Galois,
and the mathematicians continuing their work,
concentrated on 
the characterization  of the equations solvable by radicals.
The problem was clearly formulated by Abel  in his  unfinished memoir 
``Sur la resolution algébrique
des équations'' \cite{Ab}. He also indicated the way to arrive at its solution.
According to Abel one should find all equations of a given degree  that are 
algebraically solvable.
In the course of his investigations he arrives at ``several general propositions about the solvability of equations and about the form of
their roots'' \cite[p. 219]{Ab}.

The investigation of the form of the roots of a solvable equation
was taken up by Kronecker, see the interesting paper by
Birgit \href{https://zbmath.org/authors/petri.birgit}{Petri} and
Norbert \href{https://zbmath.org/authors/schappacher.norbert}{Schappacher} \cite{PS}.
Three years after his first note on the subject \cite{K1}, Kronecker observes
in \cite{K2} the following:
\begin{quote}
Wenn eine irreductible Gleichung mit ganzzahligen Coëfficienten auflösbar 
und der Grad derselben eine ungrade
Primzahl ist, so sind entweder \textit{alle} ihre Wurzeln oder nur \textit{eine} reell.
\footnote{If an irreducible equation with integer coefficients is solvable
and the degree of it is an odd
prime, then either \textit{all} of its roots or only \textit{one} are real.}
\end{quote}
This is a special case of the general result:
\selectlanguage{german}
\begin{quote}
Wenn eine Gleichung — deren Grad eine ungrade Primzahl $\mu$ ist, de\-ren
Coëfficienten rationale Functionen irgend welcher reeller Grö\-ßen  $A, B, C,  \dots$
also selbst reell sind und welche endlich nicht in Factoren niederen Grades
zerlegt werden kann, so daß deren Coëfficienten wiederum rationale
Functionen von $A, B, C,  \dots$ wären --- durch eine explicite algebraische
Function jener Größen $A, B, C,  \dots$ erfüllt wird, so sind entweder \textit{alle}
ihre Wurzeln, oder nur \textit{eine} derselben reell.
\footnote{If an equation — whose degree is an odd prime $\mu$, whose
coefficients are rational functions of any real quantities $A, B, C,  \dots$
and therefore themselves real and which finally cannot  be decomposed in 
factors of lower degree so that their coefficients are again rational
functions of $A, B, C,  \dots$ --- is satisfied by an explicit algebraic
function of those quantities $A, B, C,  \dots$ , then either \textit{all} of
its roots, or only \textit{one} of them are real.}
\end{quote}
\selectlanguage{english}
At the time Kronecker wrote this, the concept of field was not yet defined. With it
Kronecker's result can be formulated in the following way.

\begin{thm}[Kronecker's Theorem]
Let $K$ be a subfield of the real numbers.
Suppose that an irreducible
polynomial $f(X)\in K[X]$ of odd  prime degree $p$
is solvable by radicals. Then exactly one root or all roots of the polynomial are real.
\end{thm}

As Kronecker clearly states, he found his results from the study of  the form of the roots
of solvable polynomials.
This study was
continued by Heinrich Weber \cite{W},
Anders \href{https://zbmath.org/authors/wiman.anders}{Wiman} \cite{Wi}
and recently by
Harold \href{https://zbmath.org/authors/edwards.harold-m}{Edwards} \cite{E}.
In the penultimate section of the first volume of his algebra textbook \cite{W},
Weber connects the number of real roots to the
number of real roots of an auxiliary equation of degree $p-1$,
and this seems to establish  Kronecker's Theorem, but Weber
starts out from the result, referring to a simple proof earlier in his book.
This simple proof was already mentioned by Kronecker himself \cite{K2}:
\selectlanguage{german}
\begin{quote}
Ich bemerke ferner,
daß die angegebene Eigenschaft der irreductibeln auflösbaren Gleichungen $\mu$ten
Grades nicht bloß aus der allgemeinen Form ihrer Wurzeln hervorgeht, sondern
auch aus dem schon von Galois herrührenden Satze „daß jede Wurzel einer solchen
Gleichung sich als rationale Function von irgend zwei andern darstellen läßt“.
Wenn nämlich diese Function nur reelle Coefficienten enthält,  so folgt hieraus 
unmittelbar, daß alle Wurzeln reell sein müssen, sobald nur zwei derselben reell sind.
\footnote{Furthermore I remark,
that the stated property of  irreducible solvable equations of degree $\mu$
does not only  follow from the general shape of their roots, but
also from the Theorem due to Galois, “that every root of such an
equation can be represented as a rational function of any two others''.
Namely, if this function only contains real coefficients, then it follows
immediately that all roots must be real as soon as only two of them are real.}
\end{quote}
\selectlanguage{english}
Galois' result can be formulated  in modern language as follows.
\begin{thm}[Galois' Theorem]\label{galois}
An irreducible
polynomial equation of prime degree $p$ over a
subfield of $\mathbb{C}$ is solvable by radicals if and only if
its splitting field is generated by any two of its
roots.
\end{thm}

Kronecker's Theorem follows immediately, just as Kronecker remarked,
since a solvable irreducible polynomial of prime degree $p$ over a field $K\subset \mathbb R$  having at least two real roots $\alpha, \beta$ splits in the field $K(\alpha, \beta)$, so all its roots are real.

Galois proved this result in his famous
``Mémoire sur les conditions de résolubilité des équations par radicaux"
that was rejected in 1831 by
the Académie des Sciences, and only published by Joseph \href{https://zbmath.org/authors/liouville.joseph}{Liouville} in 1846, see \cite{Neu}; for a proof see \cite[Ex. 13.9]{B}. S.-D.~\href{https://zbmath.org/authors/poisson.s-d}{Poisson}, the referee of the memoir,
noted in his report (see \cite[p. 146]{Neu}) that an analogous proposition of Abel was printed after Abel's death
in \textit{Crelle’s Journal} \cite{Bruch}. The excerpt in question of Abel's letter
to A.L. Crelle from October 18, 1828 reads:

\selectlanguage{french}
\begin{quote}
Si trois racines d’une équation quelconque irréductible, dont le degré est un
nombre premier, sont liées entre elles de sorte que l’une de ces racines peut
être exprimée rationnellement par les deux autres, l’équation en question
sera toujours résoluble à l’aide de radicaux.
\footnote{If three roots of an arbitrary equation of which the degree is a prime number
are related to each other in such a way that any one of these roots may be
expressed rationally by the other two, the equation in question will
always be solvable with the help of radicals.}
\end{quote}
\selectlanguage{english}

It gives only a sufficient condition. The report goes on to say that no proof has been published and that Galois' proof of his necessary and sufficient condition is not satisfactory. Poisson complained
that the memoir did not
contain, as the title promised, the condition of
solvability of equations by radicals. For, to decide whether
a given equation of prime degree is solvable,
one should first have to determine
if this equation is irreducible, and then if one
of its roots can be expressed as a rational function of two others.
The condition, if it exists, should be verifiable by inspecting the
coefficients of a given equation, or, at most, by solving other
equations of lower degree (see \cite[p. 148]{Neu}).

Kronecker's Theorem gives a possibility
 to construct
equations that are not solvable by radicals: any irreducible polynomial equation
of odd prime degree  with integral coefficients having more than one real root and also
complex conjugate roots will do. One of the simplest examples is the
polynomial $X^5-4X-2$, see Figure \ref{figgraph}. The function $f(X)$ has only two real extrema and the polynomial is irreducible by the
Schönemann--Eisenstein criterion, published by
Th.~\href{https://zbmath.org/authors/schonemann.th}{Schönemann} in 1846 and
G.~\href{https://zbmath.org/authors/eisenstein.gotthold}{Eisenstein} in 1850 (see \cite{C});
this criterion was of course unknown to Poisson.

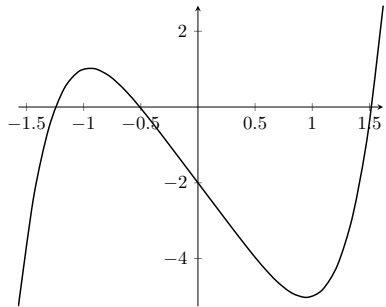
\begin{figure}[h]
\centering
\begin{tikzpicture}
  \begin{axis} [axis lines=center]
    \addplot [domain=-1.57:1.62, smooth, thick] { x^5 - 4*x -2 };
  \end{axis}
\end{tikzpicture}
\caption{The graph of $f(X)=X^5-4X-2$}
\label{figgraph}
\end{figure}

A similar example is given in Ian \href{https://zbmath.org/authors/stewart.ian-n}{Stewart}'s textbook Galois Theory \cite{St}, but on the basis
of Weber's theorem \cite[\S 186]{W}, that an irreducible polynomial $f(X)\in K[X]\subset \mathbb R[X]$
 of prime degree $p$
with $p-2$ real and 2 complex nonreal roots has as Galois group  the symmetric
group $S_p$, see \cite[Ex. 13.4]{B}. We note that Galois' Theorem
as stated above does not occur in \cite{St}, as in many other modern textbooks.


A proof of Kronecker's Theorem 
was given by Weber
in the first volume of Weber--\href{https://zbmath.org/authors/wellstein.josef}{Wellstein}'s
\textit{Encyklopädie der Elementar-Mathematik}
\cite[\S\,101]{WW1},  originally for the case $p=5$, and in  later editions for
general $p$. Elementary mathematics here means  roughly the
mathematics taught in secondary school, and 
this ``handbook for teachers and students'' is mainly directed at teachers.
A fourth, heavily revised edition \cite{WW2} was prepared by Paul 
\href{https://zbmath.org/authors/epstein.paul}{Epstein}. 
He  made almost no changes in the chapter on solvable equations, but modernized
the terminology.


Weber's proof is suitable for a university course in algebra that treats the
solution of equations, but not Galois theory. Such  a course was given
at Uppsala University by Trygve Nagell since 1931.
In his textbook \cite{N}
Nagell treats the case $p=5$ following the first edition of Weber--Wellstein.

The insolvability of the quintic was included as one of the hundred problems
in Heinrich Dörrie's book
\textit{Triumph der Mathematik} \cite{D}, intended
for a general audience. It contains a variety of problems, like Kirkman's
schoolgirl problem, or the length of the polar night, but also the transcendence of $\pi$.
Dörrie was one of the first Ph.D students of David
\href{https://zbmath.org/authors/hilbert.david}{Hilbert} with a dissertation on
the quadratic reciprocity law in  quadratic number fields with class number one.
Dörrie's book has been translated into Japanese, Hungarian, and English. The
English translation \cite{DE} is unfortunately marred by strange terminology: Körper
is not translated as field but as group, Adjunktion as substitution.
For the insolvability of equations  Dörrie follows the 1922 edition of Weber--Wellstein
\cite{WW2}.
His version is also the basis for the expositions in \cite{PC}, \cite{MJP}, and
\cite{Pr}. The last text inspired the proof in \cite{Sk}. None of the sources mentioned
cites Weber.\footnote{The latest version of the preprint \cite{PC} contains a footnote
stating that Weber earlier gave the same proof.}


\section{Preliminaries}
\subsection{What is an elementary proof?}
Elementary certainly means avoiding deep general theory. In the context of Kronecker's
Theorem this excludes the  fundamental theorem of Galois theory.
What can be used depends on the target group of the exposition. Almost a century ago Dörrie wrote
for a general audience. Consequently he started  his chapter on
Kronecker's Theorem by explaining what a field (consisting of numbers) is and 
proved the facts
about field extensions needed for Weber's proof, but avoided  the concept of
group.  The  proof uses
the theory of field extensions within the scope defined by the contents of a second course in algebra at a university level. For our notion of elementary we have 
students taking such a course in mind.
The considerations can always be limited to subfields $K$ of the complex numbers $\mathbb{C}$, which we assume throughout this article. Thus, it is possible to include field extensions by the roots of polynomials (consequently, their splitting fields),
the description of simple extensions, and the behavior of the degree according to the Tower Law. It is essentially no problem to discuss automorphism groups
of fields, since it is simply one more example of groups, which are always
discussed in such courses. Thus, the notion of the Galois group of a polynomial $f(X) \in K[X]$ as the group of all automorphisms of its splitting field $K_f=K(\alpha_1,\ldots,\alpha_n)$ that fix every element of $K$, where the $\alpha_i$ are all the roots of $f(X)=0$
in $\mathbb{C}$, may also be accepted as elementary. This group we denote by $G(K_f/K)$.  What is essential to avoid is the fundamental theorem of Galois theory (together with the Galois correspondence, in particular,
between normal subgroups and normal extensions of the ground field), and as regards
the notion of solvability, the correspondence between solvable groups and solvable extensions. But the notion of a solvable extension of fields is, of course, elementary.
It is often presented as an example of the power of (abstract) algebraic language
when one wants to formulate the problem of algebraic solvability and, in particular, when one wants to encourage the students to read a course in Galois theory in order to
 to find out why there is no formula for an algebraic solution
 of a general quintic.

We can freely  use modern concepts  such as dimension and groups
of field automorphisms, replacing the old fashioned arguments of Weber and Dörrie.
We note that not only they, but also Nagel deal with field automorphisms, looking at their actions on rational functions of the roots and 
considering  ``transformations'' of these rather complicated expressions, and therefore
with Galois groups, without mentioning them explicitly.

\subsection{Field extensions.}
In this subsection, we recall several well-known results concerning the ``elementary part'' of the theory of field extensions.  
As a general reference for  the notions and results that we mention, 
we use \cite{B}. Let $K\subseteq L$ be a field extension. An element  $\alpha\in L$ is
\textit{algebraic} over $K$ if
there is a nonzero polynomial $f(X)\in K[X]$ such that $f(\alpha)=0$. Zeros
of univariate polynomials are called \textit{roots}.
Recall that a field extension $K\subseteq L$ is called \textit{finite} if $L$ has finite dimension as a vector space  over $K$. This dimension is denoted by $[L:K]$
and is called the \textit{degree} of $L$ over $K$. The field extension
$K\subseteq L$ is \textit{simple},
if $L$ is obtained by adjoining
one element  $\alpha$, so $L$ is of the form $K(\alpha)$.
If $\alpha$ is algebraic over $K$, then the \textit{minimal polynomial}
of $\alpha$ is the monic polynomial of smallest degree with coefficients in $K$ that has
$\alpha$ as a root. For the reader's convenience and easier reference, we formulate two 
elementary properties of finite field extensions.

\begin{lem}[Simple Extension Theorem {\cite[T.4.2]{B}}]\label{lemSimpleExtension}
If $\alpha$ is algebraic  over a field $K$ and  $n$ is the degree of its minimal
 polynomial over  $K$, then $1,\alpha,\ldots,\alpha^{n-1}$ is a basis of $K(\alpha)$ over $K$. Thus $[K(\alpha):K]=n$
and each element in $K(\alpha)$ can be uniquely
  represented as $b_0+b_1\alpha+\cdots + b_{n-1}\alpha^{n-1}$, where
  $b_i\in K$.
\end{lem}

\begin{lem}[Tower Law {\cite[T.4.3]{B}}]\label{lemTowerLaw}
Let $K \subseteq L$ and $L \subseteq M$ be finite field extensions. Then $K \subseteq M$ is a finite field extension and $[M: K]=[M: L][L: K]$.
\end{lem}

A \textit{splitting field} $L=K_f$ for a polynomial $f(X) \in K[X]$ is a field $L$  that contains
all roots $\alpha_1,\dots,\alpha_n$ of $f$ and is generated over $K$ by these roots, that is,
$L=K(\alpha_1,\dots,\alpha_n)$.
Recall that a field extension $K \subseteq L$ is \textit{normal} if every irreducible polynomial
over $K$ that has one root in $L$ necessarily has all the remaining roots in $L$.
A proof that splitting fields of polynomials over a field $K$
coincide with
finite normal extensions $L$ of this field is usually not included in ``elementary'' algebra courses, but it is  straightforward and  uses only the structure of simple extensions and the Tower Law together with basic properties of isomorphism of fields  (see the proof \cite [p. 120]{B} of \cite[T.7.1]{B}).

Let $K \subseteq L$ be a finite extension. The \textit{Galois group} of $L$ over $K$,
denoted by $G(L/K)$, is the group of all automorphisms of $L$ that fix $K$ elementwise.
If $H$ is any subgroup of  $G(L/K)$, then $L^H$ is the subfield of $L$
consisting of all elements, which are fixed by all automorphism belonging to $H$.

Two properties of isomorphisms of fields play an essential role
in several arguments. The first is
often proved in the context of  a  theorem usually attributed to Kronecker,  stating that for any field
and any non-constant polynomial, there exists an extension field containing
a root of this polynomial. The second Lemma is less known, but its proof is
closely related to the first one.

\begin{lem} \label{lemIsomorphism}\textup{(a)} Let $K$ be a
field and let $\alpha, \alpha'$ be two roots
of an irreducible polynomial $f(X) \in K[X]$. Then there is an isomorphism
$\sigma\colon K(\alpha) \rightarrow K(\alpha')$ such that $\sigma(\alpha)=\alpha'$
and $\sigma$ is the identity on $K$.

\noindent\textup{(b)}
If $\tau\colon  K \to  K'$ is an isomorphism of fields, $L$ a splitting field of a polynomial
$f \in K[X]$,  and $L'$ a splitting field of the polynomial
$\tau(f) \in K'[X]$, then there exists
an isomorphism $\sigma\colon L \to  L'$ extending $\tau$ \textup{(}that is, $\sigma|_K=\tau$\textup{)}.
\end{lem}
For a proof of (a) we refer to \cite[T.5.1]{B}, and for a proof of (b) to \cite[T.5.2]{B}. It follows by induction from \cite[T.5.1]{B}, which is a more general version of (a), where
$\sigma$ extends an isomorphism $\tau$.

\begin{lem}  \label{lemIsoTwo} Let $K \subset L$ be a normal extension and let $\alpha, \alpha' \in L$
be two roots of an irreducible polynomial $f(X) \in K[X]$. Then there is an automorphism $\sigma$ of $L$ over $K$ such that $\sigma(\alpha)=\alpha'$.
 \end{lem}

We sketch the argument  from \cite[Ex. 7.4 (a)]{B}.  Since $L$ is a finite and normal extension of $K$, it is a splitting field of a polynomial, say $g(X) \in K[X]$, and also a splitting field of the same polynomial over both $K(\alpha)$ and $K(\alpha')$. By
Lemma \ref{lemIsomorphism}\,(b)
the isomorphism $\tau\colon K(\alpha) \rightarrow K(\alpha')$ over $K$ mapping $\alpha$ onto $\alpha'$ has an extension to an isomorphism of $L$.


\begin{rem}  \label{remGaloisCorr}
Notice that  Lemma \ref{lemIsoTwo}  says that the automorphism group of a normal field extension $L$ of a field $K$ acts transitively on the roots of any irreducible polynomial in $K[X]$ that has its roots in $L$. Moreover, it follows that 
$L^{G(L/K)}=K$: if $\alpha$ and $\alpha'$ are two different roots
of the minimal polynomial $f(X)$ of $\alpha$ over $K$, then 
there is an automorphism $\sigma$
of $L$ over $K$ such that $\sigma(\alpha)=\alpha'$. Therefore $\alpha$ is not contained
in the field $L^{G(L/K)}$, which of course contains $K$. This also applies to any
intermediate field $K \subseteq M \subseteq L$, since  $L$ is also normal over $M$ as the splitting field of the same polynomial whose splitting field over $K$ is equal to $L$.
This gives the ``trivial'' part of the Galois correspondence, which says that if we go from a subfield $M$ to the subgroup $G(L/M)$ and back to the corresponding field $L^{G(L/M)}$, then we come back to the same subfield $M$. In particular, this says that different subfields $M$ go to different subgroups $G(L/M)$ of $G(L/K)$.
The difficult part of the Galois correspondence
says that $G(L/L^H)=H$ for every subgroup $H$ of $G(L/K)$.
\end{rem}


We are interested in algebraic solvability of polynomials, which means expressing the roots
using arithmetical operations and radicals. The problem is best formulated in terms of
fields. A \textit{simple radical extension} is an extension $K\subseteq K(\alpha)$, where
$\alpha^q =a\in K$ for some positive integer $q$, that is, $\alpha$ is a root
of the  binomial polynomial $X^q-a$.
A field extension   $K\subseteq L$ is \textit{radical} if there is a chain
\begin{equation}\label{chain}
K=K_0 \subseteq K_1 \subseteq \cdots K_{i-1} \subseteq K_i \subseteq \cdots \subseteq K_{n}=L 
\end{equation}
of simple radical extensions $K_{i}=K_{i-1}(\alpha_i)$.
We may assume that the $q_i$ are prime numbers, by using the fact
that $\sqrt[rs]a=\sqrt[r]{\sqrt[s]a}$.
A polynomial $f(X) \in K[X]$ is \textit{solvable by radicals}, if  a splitting field of $f(X)$ is contained in a radical extension $K\subseteq L$ \cite[p. 78]{B}.

We stress the fact that  in the definition of a simple radical extension  we do not assume
that the polynomial $X^{q}-\alpha^{q}$ is irreducible over $K$.
If $X^{q}-\alpha^{q}$ is irreducible, then $\alpha=\sqrt[q]{a}$ is 
\label{pageRedIrred} 
called an \textit{irreducible radical}. Otherwise, that is, if $X^{q}-a$ is reducible
over $K$, then any $\alpha=\sqrt[q]{a}$ is called a {\it reducible radical}.
A polynomial $f(X) \in K[X]$ is \textit{solvable by irreducible radicals} if  $K_f\subseteq L$,
where $K\subseteq L$ is a radical extension such that in the chain of simple radical extensions every
$X^{q_i}-\alpha_i^{q_i}$ is irreducible over $K_{i-1}$.
From the (constructive)  point of view of writing down radical expressions for the roots
of an equation it is quite natural to  require irreducibility in the definition of
solvability. But this gives the same class of polynomials: a
solvable polynomial is also solvable by irreducible radicals, see e.g.,  \cite[Ch. 21]{St}.

In the following
sections, we need a few properties of the splitting fields of the
binomials $X^q-a$.
We gather them in the following Lemma.

\begin{lem}\label{lemBinomial} Let $K$ be a subfield of $\mathbb{C}$, $q$ a prime number, and
$a \in K\setminus\{0\}$.

\noindent\textup{(a)}  The splitting field of $X^q-a\in K[X]$
over $K$ is $K'=K(\alpha,\varepsilon)$, where $\alpha\in K'$ is a root of  $X^q-a$
and $\varepsilon$ a primitive $q$-th root of unity.

\noindent\textup{(b)} The binomial $X^q-a$ is reducible over $K$ if and only if it has a root in $K$.

\noindent\textup{(c)} If $\varepsilon \in K$ and $X^q-a$ is irreducible over $K$, then $K'=K(\alpha)$ and $[K':K]=q$. All automorphisms of $K'$ over $K$ are mappings $\varphi(\alpha)=\varepsilon^r\alpha$,
$r =0,1,\ldots,q-1$. Thus the Galois group is the cyclic group $G(K'/K)=\mathbb{Z}_q$.

\noindent\textup{(d)} If $\varepsilon \notin K$ and $X^q-a$ is reducible over $K$, then $K'=K(\varepsilon)$ is the splitting field of the polynomial $X^q-1$.
 Every automorphism of $K'$ maps $\varepsilon$
onto a power of $\varepsilon$.
\end{lem}

 We do not present detailed proofs, but we comment on the points above.

(a) The equation $X^q-a=0$ has $q$ different solutions $\varepsilon^r\sqrt[q]{a}$, where $\varepsilon \neq 1$ denotes a (primitive) $q$-th root of unity,  $r=0,\ldots,q-1$ and $\alpha=\sqrt[q]{a}$ is any fixed solution to the equation $X^q-a=0$. Thus the splitting field of the binomial $X^q-a$ over $K$ is equal to $K'=K(\alpha,\varepsilon)$ (\cite[Ex. 5.11 (c)]{B}).

(b) This is a result of Abel. The fact that the polynomial $X^q-a$ is reducible over $K$ if and only if $a$ is a $q$-th power
in  $K$ is proved in \cite[Ex. 5.12 (a)]{B}. 

(c)  If $\varepsilon \in K$ and $X^q-a$ is irreducible, then according to (b),
$\alpha \notin K$, since $a$ is not a $q$-th power in $K$ ($\alpha^q=a$). The extension $K'=K(\alpha)$ of $K$ is simple, of degree $q$ by the Simple Extension Theorem (Lemma
\ref{lemSimpleExtension}).
According to  (a), the polynomial $X^q-a$ has $q$ different  roots $\varepsilon^r\alpha$ for $r=0,1,\dots,q-1$.
Every automorphism of $K'$ maps $\alpha$ on a root of the same polynomial. Thus,
we have $q$ different automorphisms mapping $\alpha$ onto $\varepsilon^r \alpha$.
The exponents $r$ form the group $\mathbb{Z}_q$ of all residues $r$ modulo $q$ (as $\varepsilon^q=1)$.

(d) If $X^q-a$ has a root $\alpha \in K$, then it follows from (a) that $K'=K(\varepsilon)$, where as before $\varepsilon$ denotes a primitive $q$-th root of unity. 
Thus $K'$ is a splitting field of the polynomial $X^q-1$, since all $q$-th roots of unity are powers of $\varepsilon$.
The minimal polynomial of
$\varepsilon$ over $K$ divides the minimal polynomial of $\varepsilon$
over $\mathbb{Q}$,  which has degree $q-1$; it is $(X^q-1)/(X-1)$ (\cite[Ex. 4.3 (f)]{B},
for the history of this result see \cite{C}).
Every automorphism maps $\varepsilon$ onto another $q$-th
root of unity, which  is a power $\varepsilon^k$ for some $1\leq k \leq q-1$.

\subsection{Auxiliary results.} The proofs of the following results, which are needed for
Weber's and our proof of Kronecker's Theorem,
 do not depend on Galois theory.

\begin{lem}[Nagell's Lemma]\label{nagell} If
$f\in K[X]$ has prime degree $p$ and is irreducible over $K$, but reducible over
a field extension $K\subset L$, then
$p\mid [L:K]$.
\end{lem}

 \begin{proof}
Let $r=[L:K]$ and let $\alpha$ be a root of $f(X)$ in a field extension of $L$. As $f(X)$ is
reducible in $L$, we have $[L(\alpha):L]=m<p=[K(\alpha):K]$. Because $L(\alpha)$ is an
extension of $K(\alpha)$, $p=[K(\alpha):K]\mid [L(\alpha):K]=mr$, where the last equality holds by the Tower Law (Lemma \ref{lemTowerLaw}). Therefore $p\mid r$.
\end{proof}

The Lemma is formulated differently in  \cite[\S 98, p. 259]{N} and  \cite[Ex. 4.13]{B},
but the  proofs given there
establish the present result.
In fact, it follows from the general fact  that if the  degree of an irreducible polynomial in $K[X]$ and the degree of a finite field extension $L$ of $K$ are relatively prime, then the polynomial remains irreducible over $L$ (see  \cite[Ex. 4.2]{B}).
Weber \cite{WW2} and Dörrie \cite{D} use Proposition
\ref{epstein} below with the same effect.

The following  Lemma can be found in \cite[Ex. 7.8]{B}. It only depends on
Lemma \ref{lemIsomorphism} and the Tower Law.

\begin{lem}\label{equal}
Let $K \subset L$ be a normal extension and let $f(X)$ be a monic polynomial
irreducible over $K$ but reducible over $L$. Then all irreducible factors of $f(X)$ in 
$L[X]$ have the same degree.
\end {lem}

\begin{proof}
If $\alpha_i$ and $\alpha_j$ are roots (in an extension of $L$) of two irreducible factors $f_i(X)$ and $f_j(X)$ of $f(X)$ over $L$, then there is an isomorphism
$\tau\colon K(\alpha_i)\to K(\alpha_j)$ over $K$, by Lemma \ref{lemIsomorphism}\,(a) (since both $\alpha_i, \alpha_j$ are roots of $f(X)$).
Let  $L$ be a splitting field of a polynomial $g(X)\in K[X]$. Then $L(\alpha_i)$ is a splitting field of $g(X)$ over $K(\alpha_i)$, and $L(\alpha_j)$ is a splitting field of $g(X)$ over $K(\alpha_j)$. By Lemma \ref{lemIsomorphism}\,(b) this
automorphism can be extended to an isomorphism
$\sigma\colon L(\alpha_i)\to L(\alpha_j)$ over $K$.  Therefore the degrees
$[L(\alpha_i):K]$ and $[L(\alpha_j:K]$ are equal
and thus also the degrees $[L(\alpha_i):L]$ and $[L(\alpha_j):L]$ are equal, which
are just the degrees of $f_i(X)$ and $f_j(X)$ .
\end{proof}

Kronecker's Theorem is about the solvability of polynomials
with real coefficients. We need terminology and results adapted to this situation.
We say that a field $K \subseteq \mathbb{C}$ is {\it conjugation invariant} if  complex conjugation is an automorphism of this field. Notice that if $K$ is a conjugation invariant field, then its extension $K(\alpha)$, where $\alpha \in \mathbb{C}$, is conjugation invariant if and only if $\bar{\alpha} \in K(\alpha)$. 

\begin{lem}\label{LemConjInv}
Let  $K\subseteq L$ be a radical extension of  a conjugation invariant field $K$.
Then there exists a chain $K=K_0\subseteq K_1 \subseteq \cdots\subseteq K_m$ of simple
radical extensions (of prime degree) with $L\subseteq K_m$, where all fields $K_i$ are
conjugation invariant.
\end{lem}

\begin{proof}
We use induction with respect to the number $n$ of different simple radical extensions between $K$ and $L$. 
As already noted, we may assume that every simple radical extension is by a root of a binomial of prime degree.
Assume  that the Lemma is true for every radical extension of
$K$ by such a chain of at most $n \geq 0$ simple radical extensions.
The base case, when $L=K$ (that is, $n=0$), is trivially true.

Let $L$ be a radical extension of $K$ by a chain of $n+1$ simple radical extensions (each by a root of a binomial of prime degree). Let
$L'\subset L$ be the last subextension in the chain.
Then $K \subset L'(\alpha)=L$, where $\alpha^q = a \in L'$ for a prime number $q$. The field $L'$ is a radical extension of $K$ by $n-1$ simple radical extensions.
By the induction hypothesis, we have $L' \subset K_{m'}$, where
$K = K_0 \subset K_1 \subset  \cdots \subset K_{m'}$ is a chain of simple radical extensions with all fields $K_i$ conjugation invariant
and  $K_i= K_{i-1}(\alpha_i)$ with $\alpha_i^{q_i} = a_i \in K_{i-1}$ and $q_i$ a prime number for $i = 1,\ldots,m'$.

We construct a chain for $L$. We have that $L\subset K_{m'}(\alpha)$. If $\bar\alpha\in K_{m'}(\alpha)$, then
$K_{m'}(\alpha)$ is conjugation invariant and we obtain the required chain by  putting $K_{m'+1}= K_{m'}(\alpha)$.
If $\bar\alpha\notin K_{m'}(\alpha)$ we consider the real number $\rho = \alpha\bar\alpha$.
The fields $K_{m'+1}=K_{m'}(\rho)$ and $K_{m'+2}=K_{m'+1}(\alpha)=K_{m'}(\rho,\alpha)= K_{m'}(\alpha,\bar\alpha)$ are conjugation invariant.
Moreover, we have $\rho^q = \alpha^q\bar\alpha^q=a\bar a \in K_{m'}$ and $\alpha^q =a \in L' \subseteq K_{m'} \subseteq K_{m'+1}$, which shows that $K_{m'} \subseteq K_{m'+1}$
and $K_{m'+1} \subseteq K_{m'+2}$ are radical extensions by roots of prime degree binomials.
Consider now the chain
\[
K = K_0 \subseteq \cdots \subseteq K_{m'} \subseteq K_{m'+1} \subseteq K_{m'+2}\;.
\]
Clearly, we have  that $L=L'(\alpha) \subseteq K_{m'}(\alpha) \subseteq K_{m'}(\alpha,\bar\alpha)=K_{m'+2}$. If we have an equality of  fields in the chain,
we can simply remove repeating copies. The constructed chain satisfies the induction step.
\end{proof}

The number $\rho$ in the above proof plays an important role in Weber's proof. In the following Lemma we collect some properties of it, which easily follow from Lemma \ref{lemBinomial}.

\begin{lem} \label{lemRho}
Let $K$ be a conjugation invariant field and consider the prime degree
binomial $X^q-a\in K[X]$. The real root $\rho=\sqrt[q]{a\bar a}$ satisfies 
$\rho= \alpha \bar \alpha$ for any root $\alpha$ of $X^q-a$. Therefore 
$\rho \in K(\alpha,\bar\alpha)$, and if $X^q-a$ is reducible, that is, has a root in $K$,
then $\rho\in K$.
\end{lem}

\begin{rem}\label{remWeber}
Lemma \ref{LemConjInv}
is not formulated as such by Weber, but the idea of the construction is due to him.
Weber  \cite{WW1,WW2} just says that he adjoins $\alpha$ and $\bar\alpha$ at the same time. He considers this clearly
as one step, and only  when needed he first adjoins  $\rho$ and then $\alpha$. In particular, 
the first adjunction in a chain  that makes $f(X)$ reducible is an extension $L\subset
L(\alpha,\bar\alpha)$ of conjugation invariant fields, which subsequently is studied in greater detail. The same holds for Dörrie \cite{D}. The objection that sometimes has been raised
to his proof, that not all extensions in the chain are conjugation invariant, is therefore
not founded.
\end{rem}

\section{A proof of Kronecker's Theorem}\label{secOurProof}

In this section, we give a proof of Kronecker's Theorem, which, in principle,
allows us to treat both reducible and irreducible radicals using the same arguments. 
However, there is one point where a deeper knowledge of Galois theory for cyclic field extensions seems to be needed for such a unification. We comment on this after
the proof, but we choose to give
an elementary  proof, where we consider the two cases (irreducible and irreducible)
separately. For the second case we use an argument that more closely follows Weber's ideas. 
We formulate Kronecker's Theorem for polynomials with real coefficients in a conjugation
invariant field.

\begin{thm}[Kronecker's Theorem]\label{thmOurKronecker}
Let $K$ be a conjugation invariant subfield of the complex numbers.
Suppose that an irreducible
polynomial equation $f(X)\in K[X]$ of odd  prime degree $p$ with real coefficients
is solvable by radicals  over $K$. 
Then exactly one root or all roots of the polynomial are real. 
\end{thm}

\begin{proof}
We first remark that $f(X)$ has at least one real root, since
it has real coefficients. As $f(X)$ is solvable, 
we may assume by Lemma \ref{LemConjInv}  that there is a conjugation invariant chain
$K=K_0 \subseteq K_1 \subseteq \cdots  \subseteq K_{n}=L $
of simple radical extensions such that $K_i$ is an extension of $K_{i-1}$
by a root of binomial of prime degree for $i=1,\ldots, n$ and $K_f\subseteq L$.
There exists an $i$ such that $f(X)$ is irreducible in $K_{i-1}$ but reducible
in $K_i=K_{i-1}(\alpha_i)$, where $\alpha_i^q=a_i\in K_{i-1}$ and $q$ is a prime number.
For ease of notation we write $\alpha$ for $\alpha_i$ and $a$ for $a_i$.
By Nagell's Lemma \ref{nagell} the degree $p$ of $f(X)$ divides $d=[K_i:K_{i-1}]$.
We distinguish between two cases depending on whether $X^q-a$ is irreducible
 or $X^q-a$ is reducible over $K_{i-1}$.

\textit{Case I.} If $X^q-a$ is irreducible over $K_{i-1}$, then $d=q$ is prime,
which implies that
$q=p$.  We may assume that  $K_{i-1}$ contains a primitive $p$-th root of unity
$\varepsilon_p$, for otherwise we can replace $K_{i-1}$ by $K_{i-1}(\varepsilon_p)$.   Both polynomials $X^p-a$ and $f(X)$ remain irreducible over $K_{i-1}(\varepsilon_p)$ by Nagell's Lemma, as the degree $p$ is relatively prime to the degree of the minimal polynomial of  $\varepsilon_p$ over $K_{i-1}$. In fact, this polynomial divides $X^{p-1}+X^{p-2}+\cdots+1$ (the minimal polynomial of $\varepsilon_p$ over $\mathbb{Q}$), so its degree is at most equal to $p-1$. Thus $K_i=K_{i-1}(\alpha)$ is a splitting field of the polynomial $X^p-a$ (see Lemma \ref{lemBinomial}\,(c)), which says that the extension $K_{i-1}\subset K_i$ is normal. Notice that in this case $[K_i:K_{i-1}]=p$.

Since $f(X)$ is reducible over $K_i$ (and irreducible over $K_{i-1}$),
the degrees of all irreducible factors of $f(X)$  are equal by  Lemma \ref{equal}.
Thus the factors of $f(X)$ over $K_i$ are linear and all roots $\beta_1,\dots, \beta_p$ of $f(X)$  belong to $K_i$.
Hence, the splitting field $(K_{i-1})_f=K_{i-1}(\beta_1,\beta_2,\dots,\beta_p)$ of $f(X)$
over $K_{i-1}$ is a subfield of $K_i$. For every $r=1,\ldots,p$, we have $[K_{i-1}(\beta_r):K_{i-1}]=p$.
It implies that $K_{i-1}(\beta_r)=K_i$, since also $K_i$ has degree $p$ over $K_{i-1}$.
Thus $K_i=(K_{i-1})_f =K_{i-1}(\beta_r)$ for every $r=1,\ldots,p$.

Suppose that
at least two roots, $\beta_1$ and $\beta_2$, are real. By Lemma \ref{lemBinomial}\,(c)
the Galois group $G((K_{i-1})_f/K_{i-1})$ is cyclic of order $p$.
Every nontrivial automorphism has order $p$ and since the group acts transitively
on the roots $\beta_r$ (see Lemma \ref{lemIsoTwo}), its elements are cycles of length $p$.
We may choose as a
generator the cycle $\sigma=(\beta_1,\beta_2,\dots,\beta_p)$ starting with the real roots
$\beta_1, \beta_2$.
Because the powers $\beta_1^j$  form a basis of $(K_{i-1})_f$ over $K_{i-1}$ (Lemma
\ref{lemSimpleExtension}), we can write
\begin{equation}
\label{eqFirst}
  \beta_2= c_0+c_1\beta_1+\cdots+ c_{p-1}\beta_1^{p-1},
\end{equation}
with  uniquely determined coefficients $c_j \in K_{i-1}$.  Since $\beta_1$ and $\beta_2$ are real, complex conjugation gives
\begin{equation*}
  \beta_2= \bar{c}_0+\bar{c}_1\beta_1+\cdots+\bar{c}_{p-1}\beta_1^{p-1},
\end{equation*}
so $\bar{c}_i=c_i$ for $i=0,1,\ldots,p-1$.
Thus all  the coefficients $c_i$ are real.

Applying the  automorphism $\sigma$ to equation \eqref{eqFirst}, we get
\begin{equation*}
 \beta_3= c_0+c_1\beta_2+\cdots+ c_{p-1}\beta_2^{p-1},
\end{equation*}
showing that $\beta_3$ is also real. By repeating  this last argument, we find that
all roots of $f(X)$ are real, if more than one root is real.

\textit{Case II.} If $X^q-a$ is reducible, then $K_i=K_{i-1}(\alpha)=K_{i-1}(\varepsilon_q)$ for a primitive
$q$-th root of unity according to Lemma \ref{lemBinomial}\,(d). In order to simplify notation, we  denote $\varepsilon_q$ by $\varepsilon$.
Also in this case, the extension $K_{i-1}\subset K_i$ is normal as a splitting field of $X^q-1$ and $[K_i:K_{i-1}]=d > 1$.

Since the extension $K_{i-1}\subset K_i$ is normal and $f(X)$ is reducible over $K_i$ (and irreducible over $K_{i-1}$), we get as in Case I that all roots $\beta_1,\dots, \beta_p$ of $f(X)$  belong to $K_i$. Hence the splitting field
$(K_{i-1})_f=K_{i-1}(\beta_1,\beta_2,\dots,\beta_p)$ of $f(X)$
over $K_{i-1}$ is a subfield of $K_i$.
We work in the field $K_i$ instead of $(K_{i-1})_f$,
which we did in Case I.

As we know, the polynomial $f(X)$ has a real root. Assume that $\beta_1$ is real. Since $K_i=K_{i-1}(\alpha)$, we can write
\begin{equation}\label{EqRootsReducible}
\beta_1 = c_0 + c_1\alpha + \dots +c_{d-1} \alpha^{d-1}=\psi(\alpha)\;,
\end{equation}
where $\psi(X)= c_0 + c_1 X + \dots +c_{d-1} X^{d-1} \in K[X]$.  By Lemma \ref{lemRho}
we have
$\rho=\sqrt[q]{a\bar a} \in K$, and  $\rho=\alpha_i \bar\alpha_i$ for all roots
$\alpha_i=\varepsilon^i\alpha$ of $X^q-a$.
We conjugate the equality \eqref{EqRootsReducible}. We write $\bar\psi(X)$ for the 
polynomial in $K[X]$ with coefficients the complex conjugates of the coefficients of the
polynomial $\psi(X)$. As $\beta_1$ is real,
we get   $\psi(\alpha)=\overline{\psi(\alpha)}=
\bar\psi(\bar\alpha)=\bar\psi(\frac{\rho}{\alpha})$.
In the equation  $\psi(\alpha)=\bar\psi(\frac{\rho}{\alpha})$ only $\alpha$ does not lie
in $K$. 
Thus the equality  $\psi(\alpha)=\bar\psi(\frac{\rho}{\alpha})$ still holds when $\alpha$ is replaced  by another
root, which is the image of $\alpha$ by any automorphism of the field
$K_{i-1}(\alpha)=K_{i-1}(\varepsilon)$.
From Lemma \ref{lemBinomial}\,(d)
it follows that
every automorphism of this field maps $\varepsilon$ onto a power of this number, and $\alpha$ onto another root of $X^q-a$.
At the same time such an automorphism maps
$\beta_1$ onto another root of the polynomial $f(X)$. Since this polynomial
is irreducible over $K_{i-1}$, any root can be mapped at any other root (see Lemma \ref{lemIsoTwo}), so the
equality $\beta_r =\bar\beta_r$ is valid for all roots of $f(X)$, showing
that they are real numbers.
\end{proof}

\begin{rem} \label{remCyclicUnique} In the proof above we have the situation that all roots $\beta_1,\dots,\beta_p$
of an irreducible polynomial $f(X)\in K[X]$ generate subfields $K(\beta_i)$ of an extension field  $L\supset K$
of the same degree $p$  over $K$. In Case I, we conclude that all these fields are equal
(and equal to $L$), since also
$L$ has degree $p$ over $K$. In case II, it is not possible to claim
the equality of all subfields $K(\beta_i)$ using the same argument.
But assuming a little more knowledge than we needed in Case II, we can establish
equality.
Also in this case, the extension $K \subset L$ is normal and cyclic, so
the Galois group is cyclic. A known property of finite cyclic groups says that for every divisor $\delta$ of the group order, there is exactly one subgroup whose order is equal to $\delta$
(see \cite[Prop. A.2.2 (b)]{B}). Since for any two roots $\beta_i$ and $\beta_j$
there is an automorphism $\sigma$ of $L$ over $K$ such that $\sigma(\beta_i)=\beta_j$ (see Lemma \ref{lemIsoTwo}), the groups $G(L/K(\beta_i))$ and
$G(L/K(\beta_j))$ are isomorphic (in fact, conjugated by $\sigma$ --- see \cite[Prop. A.9.2]{B}), so they have the same order $\delta$.
Thus the group $H=G(L/K(\beta_i))$ is the same for all $i$.
The ``trivial part'' of Galois correspondence (see Remark \ref{remGaloisCorr}) says that for intermediate
fields $M$ of a Galois extension $K \subseteq L$, we have $L^{G(L/M)}=M$.
Applying this to  $M=K(\beta_i)$,  we get $K(\beta_i)=L^H$.
Thus all the fields $K(\beta_i)$  corresponding to the roots
of $f(X)$ are equal, and we can use the argument of Case I.
Treating in this way both cases simultaneously results in a unified  and short proof of Kronecker's Theorem.
\end{rem}


\section{Weber's proof.}

In this Section, we return to the history of ``elementary'' proofs and discuss the errors
that such proofs had until recently, and  proposed corrections.

As our proof is inspired by the presentation of Nagell \cite{N}, it follows the main lines
of Weber's proof.  Given a chain  $K=K_0 \subset K_1 \subset \cdots  \subset K_{n}=L $  of simple radical extensions, the essential point is to study what happens in
the first extension $K_{i-1} \subset K_i$ such that the solvable
polynomial $f(X)$ of degree $p$
is irreducible in $K_{i-1}$ but reducible in $K_i$, where $K_i=
K_{i-1}(\alpha)$ with $\alpha$ a root of the binomial $X^q-a$, $q$ prime.
Weber (and Epstein) \cite{WW1,WW2} as well as Nagell \cite{N}
claim that $q=p$, but this does not hold if $X^q-a$
 is reducible over $K_{i-1}$.

The main point in
Weber's proof  (for the irreducible case) that differs from our proof, is the  following. As before, it can be assumed
that $\varepsilon:=\varepsilon_p \in K_{i-1}$. The real root $\beta_1$ can be written as
\begin{equation}\label{EqRoots}
\beta_1 = c_0 + c_1\alpha + \dots +c_{p-1} \alpha^{p-1}=\psi(\alpha)\;,
\end{equation}
where $\psi(X)=c_0 + c_1X + \dots +c_{p-1} X^{p-1}\in K_{i-1}(X)$
(Lemma \ref{lemSimpleExtension}).
Weber  distinguishes two cases depending on whether $\alpha \in \mathbb{R}$ or
$\alpha \notin \mathbb{R}$.

If $\alpha \in \mathbb R$,
then complex conjugation gives
that $\bar c_i=c_i$, or in other words, all $c_i$ are real,
by uniqueness  of  the coefficients in equation \eqref{EqRoots}.
As the automorphisms of the field $K_i$ over $K_{i-1}$ map $\alpha$ onto $\varepsilon^r\alpha$
for $r=0,\dots,p-1$, the other roots of $f(X)$ are
\[
\beta_r = c_0 + c_1\varepsilon^r\alpha + \dots +c_{p-1} \varepsilon^{r(p-1)}
\alpha^{p-1}\;,
\]
so they come in complex conjugate pairs and $\beta_1$ is the only real root.

If $\alpha \notin \mathbb R$,  then consider $K_{i-1}(\rho)$ with $\rho
=\sqrt[p]{a\bar a}$ (by construction $K_{i-1}$ is conjugation invariant, see Remark
\ref{remWeber}). If $f(X)$ is reducible in $K_{i-1}(\rho)$, we are back in the previous case. 
Otherwise $f(X)$ becomes reducible in the extension
$K_{i-1}(\rho)\subset K_{i-1}(\rho,\alpha)$. Then  complex conjugation gives
 $\beta_1=\psi(\alpha)=\bar\psi(\bar\alpha)=\bar\psi(\frac{\rho}{\alpha})$ and
Weber argues that, as
$\bar\alpha_s = \frac{\rho}{\alpha_s}$ for any root of $X^p-a$, the equality
$\psi(\alpha)=\bar\psi(\bar\alpha)$ still holds when $\alpha$ is replaced by another root when we apply the automorphism mapping $\beta_1$ onto this root (that is,
we map $\alpha$ onto $\varepsilon^s\alpha$). Therefore all roots are real.

The wrong claim $q=p$ is based on an application of
the following variant of Nagell's Lemma \ref{nagell}, which is proved by Weber
\cite[\S\,111, 5]{WW2}  and Dörrie  \cite[\S\,24, Satz IV]{D}.
Interestingly, the proof of this relatively elementary result by both Weber and Dörrie  does not use the concept of the degree of a field extension, which makes it relatively long and somewhat complicated. Nagell \cite{N} gives a short proof of  Lemma \ref{nagell}, using similar arguments as we do.

\begin{prop}\label{epstein}
An irreducible polynomial $f(X)$  of prime degree $p$ can
only become reducible through adjunction of a root of an \emph{[irreducible]} equation
 $\varphi(X)=0$,
whose degree is divisible by $p$.
\end{prop}

We write the word irreducible in brackets, because it occurs in \cite{D}, but
it is left out in \cite{WW2}, although the proof  in \cite{WW2} uses the assumption
that the equation $\varphi(X)=0$ is irreducible.
In \cite{WW2} the Proposition is applied to the  equation $X^{q}=a$,
with the conclusion that $q=p$. Irreducibility of this equation is not assumed,
not even implicitly, as the next step  argues that $a$ is not the $p$-th power of a number
$\beta \in K_{i-1}$, because otherwise  $\alpha$ has the form  $\varepsilon_p^k\beta$ and would belong to
$K_{i-1}$.  This is clearly a circular argument.

Dörrie \cite{D} apparently noticed the error.
To be able to use the correctly stated Proposition
\ref{epstein} he  requires the radicals in the chain to be
irreducible.
Dörrie does not realize that he might introduce reducible radicals when
making the chain conjugation invariant:
\begin{quote}
Ferner wollen wir mit jedem adjungierten Radikal unsere Kette, das noch
keine Zerlegung von $f(x)$ ermöglicht, auch gleich das komplex konjugierte Radikal
adjungieren. Das ist vielleicht überflüssig, sicher aber nicht schädlich.\footnote{Furthermore, with each adjoined radical of our chain,
which not yet makes it possible to factorize $f(x)$, we will also adjoin
at the same time the complex conjugate
radical. That is maybe superfluous, but certainly not harmful. }
\end{quote}
But actually harm can be done, as the following example shows.

\begin{ex}\label{vdmonde}
Let $\zeta$ be a primitive 7-th root of unity and let $f(X)$ be the minimal polynomial
of the real number $\beta=\zeta+\bar\zeta$. One easily computes that
\[
f(X)=X^3+X^2-2X-1\;.
\]
In fact, $\beta$ generates the maximal real subfield
of the cyclotomic field $\mathbb{Q}(\zeta)$ whose degree over $\mathbb Q$ is equal to 6.
Thus its maximal real subfield has degree 3 over $\mathbb{Q}$. A similar construction
is possible for any prime number $p$ such that $2p+1$ is also a prime number (Sophie \href{https://zbmath.org/authors/germain.sophie}{Germain} primes).

Let a chain  of simple radical extensions start
with $K_0=\mathbb{Q} \subset K_1=\mathbb Q(\alpha)$, where e.g., $\alpha =
\zeta\sqrt[7]{11}$ and  $\sqrt[7]{11}$ is the real root. Then $\alpha$ is an irreducible radical (the root of $X^7-11=0$)
and the polynomial $f(X)$ of degree 3 is irreducible over $K_1$  by  Nagell's Lemma
\ref{nagell}, as 3 does not divide  the degree $[K_1:\mathbb{Q}]=7$.
But when we extend $K_1=\mathbb Q(\alpha)$ by the complex conjugate $\bar\alpha= \zeta^{6}\sqrt[7]{11}$, then $\bar\alpha$ is a reducible radical over $K_1$
(it satisfies the equation $X^7-11=0$ and the polynomial $X^7-11$ is reducible over $K_1$).
Moreover, we have $\zeta \in K_2 =\mathbb Q(\alpha,\bar\alpha)$ as
$\alpha/\bar\alpha=\zeta^2$ and $\zeta^2$ is a primitive  $7$-th root of unity.
Hence $\beta = \zeta + \bar\zeta \in \mathbb{Q}(\alpha,\bar\alpha))$ and $f(X)$ is reducible over this field. According to Weber and Dörrie and the construction in Lemma
\ref{LemConjInv} we should in this case first
adjoin $\rho=\alpha\bar\alpha=\sqrt[7]{121}$ and then $\alpha$, but $X^7-11$
is also reducible over $\mathbb Q(\sqrt[7]{121})$, as $\sqrt[7]{11}=(\sqrt[7]{121})^4/11$
is a root in $\mathbb Q(\sqrt[7]{121})$.

It is possible to express $\beta$ in terms of $\alpha$, which Case II in the proof of 
Theorem \ref{thmOurKronecker} suggests, but it is easier to adjoin $\zeta$
in the first place, something that is allowed according to Weber  and Nagell. The 
polynomial $f(X)$ becomes reducible by the conjugation invariant extension
$\mathbb Q \subset\mathbb{Q}(\zeta)$ of degree 6. 
The automorphism  $\sigma_k$ of $\mathbb{Q}(\zeta)$ that sends $\zeta$ to $\zeta^k$, maps $\beta = \zeta+\zeta^6$ to 
$\beta_k=\zeta^k+\zeta^{7-k}$, for $k=1,\dots,6$.
 As $\beta_k=\beta_{7-k}$, we obtain
all three  roots of $f(X)$ twice.

It is easy to solve the equation $f(X)=0$ using a computer algebra system.
The three roots are
\[
\frac{\sqrt[3]{28+84\sqrt{-3}}}6 + \frac{14}{3\sqrt[3]{28+84\sqrt{-3}}}+\frac13\;,
\]
where one should always take the same values of the multiply occurring square and
cubic roots. This answer avoids the ambiguities in Cardano's formulas.
We note that the  primitive  $7$-th root of unity $\zeta$ is a root of the polynomial
$X^2-\beta X +1$, but we refrain from giving an explicit expression for $\zeta$.
\end{ex}

Weber's and Dörrie's  argument can be rescued if the adjunction of a reducible
radical making the polynomial $f(X)$ reducible is replaced by a chain of adjunctions
of irreducible radicals, using Gauss'  result that the roots of unity are expressible by irreducible radicals \cite[Ch.~21]{St}.
This is  the approach chosen by Yan Pan and Yuzhen Chen \cite{PC} and A.B.~%
\href{https://zbmath.org/authors/skopenkov.arkadij-b}{Skopenkov} \cite{Sk},
who pointed out 
an error related to reducibility  in the proof by V.V.~%
\href{https://zbmath.org/authors/prasolov.victor-v}{Prasolov} \cite{Pr}.
Instead of the conjugation invariant chain provided
by our  Lemma \ref{LemConjInv} they construct a chain where in addition
 all prime degree binomials are irreducible (\cite[Thm. 5.1]{PC} and
\cite[Lemma 8.4.15]{Sk}).
 The proof is by induction, using the fact that there exists
for any prime  $p\geq 3$ numbers  $\beta_0,\dots,\beta_{p-2}$ such that
$\beta_i\in \mathbb Z[\varepsilon_{p-1}]$ with $\varepsilon_{p-1}$ a primitive $(p-1)$-th
root of unity and
$\varepsilon_p\in \mathbb Q(\varepsilon_{p-1}, \beta_0,\dots,\beta_{p-2})$
(see \cite[(21.9)]{St}). 
In Example \ref{vdmonde} the construction of a chain of adjunctions of irreducible
ideals would lead to the adjunction of the radicals expressing the primitive root
of unity $\zeta$.

The complications caused by making all radicals irreducible can be avoided by using the
same argument in the reducible case as in our proof of Theorem \ref{thmOurKronecker}.
In fact, this argument is similar to the one Weber uses to treat equation \eqref{EqRoots}
in  the case that $f(X)$ is still reducible over $K_{i-1}(\rho)$, although more complicated. 
In the irreducible case we have 
$\varepsilon_p\in K$ and  it is immediate from \eqref{EqRoots} that varying the roots of 
$X^p-a$, that is, multiplying $\alpha$ with powers of $\varepsilon_p$, gives $p$ different
roots. In the reducible case adjunction of $\alpha$ also  adjoins $\varepsilon_q$ 
and typically $p<d\leq q$,  as in  example 
\ref{vdmonde}. The fact that all roots of $f(X)$ are obtained follows from Lemma 
\ref{lemIsoTwo}. Of course this Lemma can also be applied in the easier case $q=p$. 
Therefore both cases, where all roots are real, can be combined. 
They are characterized by the properties of the number $\rho = \sqrt[q]{a \bar a}$.

\begin{prop} 
Let $f(X)\in K(X)$ be an irreducible polynomial of odd prime degree $p$
with real coefficients lying
in a conjugation invariant field $K$  and let $K\subset L=K(\alpha)$, where $\alpha^q=a$, $a \in K$, be a normal and conjugation invariant
extension such that $f(X)$ is reducible over $L$. In case $q=p$, assume that $K$
contains a primitive $p$-th root of unity $\varepsilon$. 
Let $\rho = \sqrt[q]{a \bar a} =\alpha\bar\alpha$. 
If $\rho\in K$, then all roots of $f(X)$ are real, while if $\rho\notin K$, then $f(X)$ has 
exactly  one real root.
\end{prop}

\begin{proof}[Sketch of proof]
If $\rho \in K$, then we write
the real root $\beta_1$ as in equation \eqref{EqRootsReducible}, where the case
$q=p$ is allowed. The argument for Case II in the proof of Theorem \ref{thmOurKronecker}
applies, as remarked above. If $\rho\notin K$, then $X^q-a$ is irreducible, and $q=p$.
As $\rho\in K(\alpha)=K(\alpha,\bar\alpha)$ we have  $K(\alpha)=K(\rho)$
by the Tower Law (Lemma \ref{lemTowerLaw}).
Therefore we may replace $\alpha$ by the real number $\rho$, and Weber's argument for
the case $\alpha\in \mathbb R$ applies.
\end{proof}

As Kronecker \cite{K2} already remarked, the two cases, only one or all roots real, are for $p\equiv 3 \pmod 4$
distinguished by the sign of the discriminant of the polynomial (the square of the product
of all differences between roots). This follows from the fact that the discriminant
of a polynomial with real coefficients (without multiple factors) is positive if the number of pairs of complex conjugate roots is even and negative
if this number is odd.  It also follows that for $p\equiv 1 \pmod 4$ only irreducible
equations with positive discriminant can be solvable.

\section*{Acknowledgements}
We express our gratitude to the referee for all remarks and comments, which
contributed to a significant improvement of the paper.


\end{document}